\documentclass[11pt]{article}

\usepackage[margin=1in]{geometry}
\usepackage{amsmath}
\usepackage{amsfonts}
\usepackage{amssymb}
\usepackage{amsthm}
\usepackage{mathtools}
\usepackage{graphicx}
\usepackage{xcolor}
\usepackage{enumerate}
\usepackage{tikz}\usetikzlibrary{decorations.pathreplacing}
\graphicspath{ {./figures/} }

\newtheorem{lemma}{Lemma}
\newtheorem{corollary}{Corollary}
\newtheorem{prop}{Proposition}
\newtheorem{theorem}{Theorem}
\theoremstyle{definition}
\newtheorem{definition}{Definition}
\usepackage{biblatex}

\renewcommand{\epsilon}{\varepsilon}
\newcommand{\abs}[1]{\left\lvert#1\right\rvert}
\newcommand{\norm}[1]{\left\lVert#1\right\rVert}

\newcommand{\Z}{\mathbb{Z}}
\newcommand{\R}{\mathbb{R}}
\newcommand{\gabor}{\mathcal{G}}
\newcommand{\defeq}{\vcentcolon=}

\title{The frameset of the second-order B-spline along $ab=1/q$}
\author{Alexander Stangl and Christina Frederick \\[4pt]
\normalsize Department of Mathematical Sciences,\\ New Jersey Institute of Technology, Newark, NJ, USA \\
\normalsize \texttt{ajs282@njit.edu} \quad \texttt{christin@njit.edu}}

\begin{document}
\maketitle
\begin{abstract}
We give a complete characterization of the zero set of the Zak transform of the second-order B-spline $Q_2$. As a consequence, we obtain a new proof of the Lemvig--Nielsen obstructions to the Gabor frame property of $Q_2$, and, along the hyperbola $ab=1/q$, we show these obstructions are sharp: every point they do not exclude belongs to the frame set. This gives the first complete characterization of $\mathcal F(Q_2)$ along an entire curve in this family, resolving a question left open by Lemvig and Nielsen in \cite{Lemniel}.
\end{abstract}
\section{Introduction}
The \textit{Gabor system} generated by a non-zero window function
\(g \in L^{2}(\R)\) along the lattice $a\mathbb{Z}\times b\mathbb{Z}$, $a,b>0$, defined by
\begin{equation}\label{eq:Gabor_glambda}
	\mathcal{G}(g, a,b) = \{g(\cdot-ak) 
		e^{2\pi imb\cdot}:\, m,k \in \mathbb{Z}\},
\end{equation} is said to be a {\it frame} provided that there exist constants $0<A\leq B< \infty$  such that for each \(f \in
L^2(\R)\)
\begin{equation}
	A\norm{f}^2 \leq \sum_{m,k \in \mathbb{Z}}
		\abs{\langle{f, g(\cdot-ak) 
		e^{2\pi imb\cdot}\rangle}}^2 \leq B\norm{f}^2.
\end{equation}

Here we consider Gabor systems generated by B-splines of order $n=2$ or higher, which are  even, compactly supported functions defined
recursively as
\begin{equation}\label{eq:sn}
	Q_1(x) = \chi_{[-1/2, 1/2]}(x), \qquad Q_n(x) = Q_1(x) * Q_{n-1}(x).
\end{equation}

For $n\geq 2$, it is an open problem to fully characterize the 
\textit{frame set} 
\(\mathcal{F}(Q_n)\),
\begin{equation}
	\mathcal{F}(Q_n) = \left\{\;(a,b) \in \R^2 : \mathcal{G}(Q_n, a,b)
		\text{ is a frame over } L^2(\R)\;\right\}.
\end{equation}

The frame set of $Q_n$ for $n\geq 2$ was posed as an open problem by Christensen \cite{christensen2014six}, and remains unresolved outside a small number of special cases. See
\cite{AtiKouOko1, atindehou2023frame, Olehon, Olehon1, Ole4, ghosh2025gabor, ghosh2023obstructions, Grojan, Lemniel, Kloosto} for recent progress on finding new regions of $\mathcal{F}(Q_n)$. As noted above,
 \(\mathcal{F}(Q_2)\) is an open set in $\{(a, b) \in \R^2_+: \, 0<ab<1\}$ \cite{FeiKai04},
but a full characterization of this set remains an open question.

 In \cite{Lemniel}, the authors provided the first class of counterexamples to the frame set conjecture for the B-spline that rely on nontrivial geometric constraints. For $x\in \mathbb{R}$ we let $R(x) = \lfloor x+\frac{1}{2} \rfloor$ and $F(x) = x-R(x)\in [-\frac{1}{2}, \frac{1}{2})$.

\begin{theorem}[Lemvig--Nielsen obstructions, Theorem 7 in \cite{Lemniel}]\label{thm:lemneil} Let $n\in \mathbb{N}$, let $a>0$, $b>3/2$, and $p,q\in \mathbb{N}$ where $\gcd(p,q)=1$.  If \(ab = \frac{p}{q}<1\), \(\abs{F(b)}
    \leq \frac{1}{nq}\),
    then \(\gabor(Q_n,a,b)\) is not a frame for \(L^2(\R).\)
    
\end{theorem}
In the last decade, there has been progress in finding counterexamples, e.g., $(a,b)\not \in \mathcal{F}(Q_2)$ \cite{grochenig2015partitions, lemvig2026new, ghosh2025obstructions}. These analyses are based on the Zak Transform and the Zibulski-Zeevi representation of Gabor systems \cite{zibulski1997analysis}. In particular, the authors show that the Zibulski-Zeevi matrix is nonsingular, e.g., by showing zero row failure or exhibiting linearly independent vectors.  Further obstructions have been found using a partition of unity property. In \cite{atindehou2023frame}, the authors use a different approach involving linear systems generated by the dual frame relationship.  In \cite{ghosh2025gabor}, the authors broadened the scope to totally positive functions and Hermite functions generally.

Each of these results, however, identifies points or subsets of curves \emph{excluded} from the frame set; whether the remaining points on any such curve actually lie \emph{in} the frame set has remained open. Here we resolve this question along the full hyperbola $ab=1/q$, $b\geq 3/2$: using the same Zak transform and Zibulski--Zeevi framework, together with a complete characterization of the zero set of $Z_{1/b}Q_2$, we establish both directions, giving the first full characterization of $\mathcal F(Q_2)$ along an entire curve in this family.

\subsection{Contributions and organization}

Our main technical contribution is a complete characterization of the zero set of the Zak transform $Z_{1/b}Q_2$ (hitherto characterized in part in \cite{janssen1988, Lemniel}): for every $(x,\gamma)$, we determine exactly when $Z_{1/b}Q_2(x,\gamma)=0$, in terms of two explicit families,  whole-interval zeros at $\gamma=k/R(b)$, and a single isolated zero at $x=\tfrac12$ governed by a transcendental equation. This builds on the result of \cite{Lemniel} (Proposition~6), which identifies zeros only under the hypothesis $|F(b)|\le\tfrac{1}{nq}$ and does not address the complementary regime or the isolated zero family at $x=\tfrac12$.

\begin{theorem}\label{thm:zerosZak}
Let $b\geq\frac{3}{2}$, $N=R(b)$, and $\epsilon=F(b)$. For $(x,\gamma)\in[0,1)^2$, $(Z_{1/b}Q_2)(x,\gamma)=0$ only if $x\in[|\epsilon|,1-|\epsilon|]$, and moreover only if either
\begin{enumerate}[(i)]
    \item $\gamma=k/N$ for some $k\not\equiv0\pmod N$, in which case 
    \begin{align}
        (Z_{1/b}Q_2)(x,k/N)=0 \text{ for every } x\in[|\epsilon|,1-|\epsilon|]; \label{eq:zakrootsieps}
    \end{align}
 or
    \item $\gamma$ satisfies
    \begin{align}
        \tfrac12\cot(\pi\gamma) + (b-N)\cot(N\pi\gamma) = 0, \label{eq:coteqn}
    \end{align}
    in which case $x=\tfrac12$ is the unique zero.
\end{enumerate}
\end{theorem}

Figure \ref{fig:zeros} shows a plot of the zeros given by Theorem \ref{thm:zerosZak} for $\tfrac{3}{2} \leq b\leq \tfrac{15}{2}$.
    
    \begin{figure}[t]
        \centering
        \includegraphics[width=.8\textwidth]{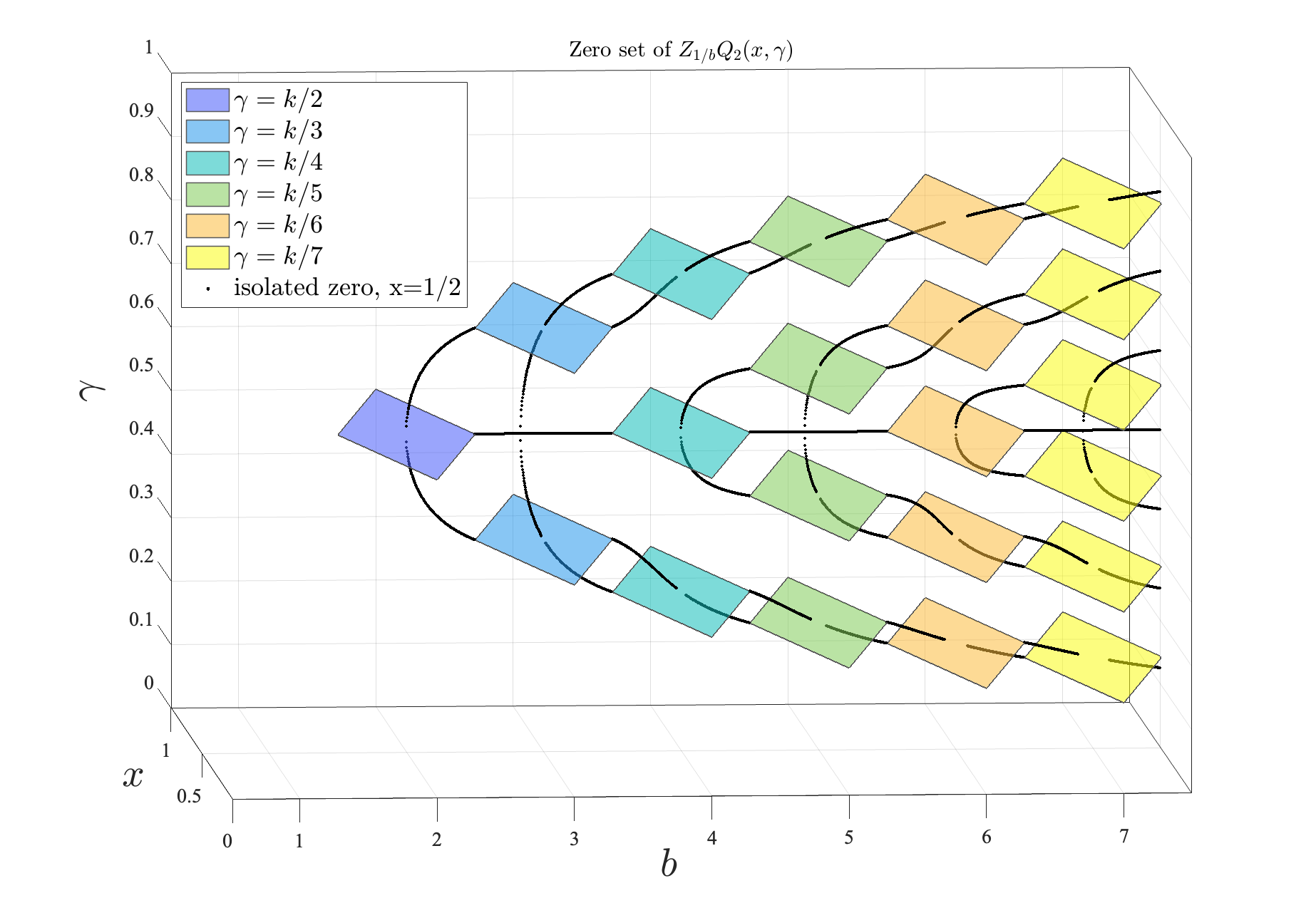}
        \caption{Roots of $Z_{\frac{1}{b}}Q_2$ given by Theorem \ref{thm:zerosZak}.}\label{fig:zeros}
    \end{figure}

As a direct consequence, the Lemvig--Nielsen obstruction for $n=2$ follows by the same zero-row argument used in \cite{Lemniel}; we include this derivation in \S \ref{sec:alternateproof} to show it fits naturally within the present framework before turning to our main result: that these obstructions are the only ones on the hyperbolas $ab=\frac1q$, $b\geq\frac32$, i.e., every point not excluded by Theorem \ref{thm:lemneil} lies in $\mathcal F(Q_2)$. This resolves, for $n=2$, the question left open in \cite{Lemniel}, where the authors state that they believe their obstruction range is optimal for $p=1$ but support this only by analogy to the first-order case (\cite{Lemniel}, Remark following Theorem~7).

\begin{theorem}\label{thm:main} Let $q$ be a positive integer. Along the curve \(ab = \frac{1}{q}\), for \(b \geq \frac{3}{2}\),
    \(\gabor(Q_2,a,b)\) is a frame if and only if \(\abs{F(b)} >
    \frac{1}{2q}\).
\end{theorem}

This result is specific to $p=1$. The obstruction direction of Theorem~\ref{thm:lemneil} holds for general $p$, since the zero-row argument does not depend on $p=1$. The converse for $p>1$ remains open: it requires a uniform lower bound on the smallest singular value of the full $p\times q$ Zibulski--Zeevi matrix, rather than the zero set of the single scalar function $Z_{1/b}Q_2$, and we leave this to future work.

The remainder of this paper is organized as follows. Section~\ref{sec:prelim} recalls definitions and preliminaries underlying our approach. Section~\ref{sec:roots} contains our main technical contribution: Lemma~\ref{lem:zakexpr} gives a closed-form expression for the Zak transform of $Q_2$, and Theorem~\ref{thm:zerosZak}, proved in Section \ref{sec:charall}, gives a complete characterization of its zero set. As a first application, Section \ref{sec:alternateproof} gives a new proof of the Lemvig--Nielsen obstructions (Theorem~\ref{thm:lemneil}) for $n=2$. Section~\ref{sec:main} proves our main result, Theorem~\ref{thm:main}, establishing that these obstructions are sharp along $ab=1/q$.

\subsection{Notation and Preliminaries}\label{sec:prelim}

\begin{definition}
The \emph{Zak transform} of $f\in L^2(\mathbb R)$ with parameter $\lambda>0$ is
\[
    (Z_\lambda f)(x,\gamma) = \sqrt\lambda \sum_{k\in\mathbb Z} f(\lambda(x-k))\, e^{2\pi ik\gamma}, \qquad \text{a.e. } (x,\gamma)\in\mathbb R^2,
\]
with convergence in $L^2_{loc}(\mathbb R^2)$. This normalization makes $Z_\lambda$ a unitary map of $L^2(\mathbb R)$ onto $L^2([0,1)^2)$, satisfying
\[
    Z_\lambda f(x+1,\gamma) = e^{2\pi i\gamma}Z_\lambda f(x,\gamma), \qquad Z_\lambda f(x,\gamma+1) = Z_\lambda f(x,\gamma),
\]
so that $[0,1)^2$ is a fundamental domain for $|Z_\lambda f|$, independent of $\lambda$.
\end{definition}

We consider {\it rationally oversampled} Gabor systems, i.e.,

\[ ab=\frac{p}{q}, \qquad p,q \text{ positive integers with} \gcd(p,q)=1.\]

\begin{definition}[Zibulski--Zeevi matrix]\label{def:ZZ} The $p\times q$ \emph{Zibulski--Zeevi matrix} corresponding to the rationally oversampled Gabor system $\mathcal{G}(g,a,b)$ generated by a function $g\in L^2(\mathbb{R})$ is the matrix $\Phi^g(x,\gamma) = \big[\phi_0^g(x,\gamma)\ \cdots\ \phi_{q-1}^g(x,\gamma)\big]$ with column vectors $\phi_s^g(x,\gamma)\in\mathbb C^p$, $s=0,\ldots,q-1$, given by
\[
    \phi_s^g(x,\gamma) = \left(p^{-1/2}(Z_{1/b}g)\!\left(x-s\tfrac pq,\ \gamma+\tfrac rp\right)\right)_{r=0}^{p-1}, \qquad \text{a.e. } (x,\gamma)\in\mathbb R^2.
\]

\end{definition}

\begin{theorem}[Zibulski--Zeevi characterization \cite{zibulski1997analysis}]\label{thm:ZZchar}
Let $A,B>0$ and let $g\in L^2(\mathbb{R})$. Then, the rationally oversampled Gabor system $\mathcal G(g,a,b)$ generated by a function $g\in L^2(\mathbb{R})$ is a frame for $L^2(\mathbb R)$ with bounds $A,B$ if and only if $\{\phi^g_s(x,\gamma)\}_{s=0}^{q-1}$ is a frame for $\mathbb C^p$ with the same bounds, for a.e.\ $(x,\gamma)\in[0,1)^2$.
\end{theorem}

The Zibulski--Zeevi matrix is continuous for Gabor systems generated by $g=Q_n$, and therefore, to find counterexamples, it suffices to find a single point $(x,\gamma)\in [0,1)^2$ such that $\{\phi_{s}^g(x,\gamma)\}_{s=0}^{q-1}$ is not a spanning set.


\section{Roots of $Z_{1/b}Q_2$ }\label{sec:roots}
In this section we derive analytic formulas for $(Z_{1/b}Q_2)(x,\gamma)$ and characterize its roots.

Let $0<a<1$, $b\geq\frac{3}{2}$ and let $q$ be a positive integer such that $ab=\frac{1}{q}$.  Throughout this paper, we set \[N=R(b),\qquad \epsilon=F(b).\]

\begin{lemma}\label{lem:zakexpr}
For $\gamma\in(0,1)$ and $x\in[0,1)$,
\begin{align}
    (Z_{1/b}Q_2)(x,\gamma) &= b^{-3/2}\Bigg[\frac{(b+x+\ell)\,\omega^{\ell} - 2x\, - (b-x-n)\,\omega^{n}}{1-\omega}+ \frac{\omega\,(\omega^{\ell} - 2 + \omega^{n})}{(1-\omega)^2}\Bigg], \label{b2zak}
 \end{align}
where $\omega=e^{-2\pi i\gamma}$, and $\ell$ and $n$ are functions of $x$ given by
\[
    \ell(x) = \left\lceil -(x+b)\right\rceil, \qquad
    n(x) = \left\lceil b-x\right\rceil.
\]
\end{lemma}
\begin{proof}
    Noting that
    \[
        (Z_\lambda f)(x,\gamma) = \frac{1}{\sqrt{\lambda}} \sum_{k \in \Z} f(\lambda(x - k)) e^{2\pi ik\gamma} = \frac{1}{\sqrt{\lambda}} \sum_{k \in \Z} f(\lambda(x + k)) e^{-2\pi ik\gamma}
    \]
    and splitting the latter sum into those points which lie on either branch of the second order B-spline, we obtain two sums over arithmetico-geometric sequences, which may be evaluated as
    \[
        \begin{aligned}
        (Z_{\frac{1}{b}} Q_2)(x,\gamma) &= \frac{1}{\sqrt{b}} \sum_{k \in \Z} Q_2(\tfrac{1}{b}(x + k)) e^{-2\pi ik\gamma}\\
        &= \left(\frac{1}{\sqrt{b}} \sum_{k = \ell}^{m-1} Q_2(\tfrac{1}{b}(x + k)) e^{-2\pi ik\gamma}\right) + \left(\frac{1}{\sqrt{b}} \sum_{k = m}^{n-1} Q_2(\tfrac{1}{b}(x + k)) e^{-2\pi ik\gamma}\right)\\
        &= \frac{1}{\sqrt{b}}\frac{\left(1 + \frac{1}{b}(x + \ell)\right) e^{-2\pi i\ell \gamma} - \left(1 + \frac{1}{b}(x + m)\right) e^{-2\pi im\gamma}}{1 - e^{-2\pi i\gamma}}\\
        &+ \frac{1}{\sqrt{b}}\frac{\frac{1}{b} e^{-2\pi i\gamma}(e^{-2\pi i\ell\gamma} - e^{-2\pi im\gamma})}{(1 - e^{-2\pi i\gamma})^2}\\
        &+ \frac{1}{\sqrt{b}}\frac{\left((1 - \frac{1}{b}(x + m)\right)) e^{-2\pi im\gamma} - \left((1 - \frac{1}{b}(x + n)\right) e^{-2\pi in\gamma}}{1 - e^{-2\pi i\gamma}}\\
        &- \frac{1}{\sqrt{b}}\frac{\frac{1}{b} e^{-2\pi i\gamma}(e^{-2\pi im\gamma} - e^{-2\pi in\gamma})}{(1 - e^{-2\pi i\gamma})^2}\\
        &= b^{-3/2} \left(\frac{(b + x + \ell) \omega^\ell - 2(x + m) \omega^m - (b - x - n) \omega^n}{1 - \omega} + \frac{\omega (\omega^\ell - 2\omega^m + \omega^n)}{(1 - \omega)^2}\right).
        \end{aligned}
    \]
    The appropriate splitting is obtained from conditioning \(\ell\), \(m\), and \(n\) on the inequalities
    \[
        \frac{1}{b}(x + \ell) \geq -1, \qquad
        \frac{1}{b}(x + m) \geq 0, \qquad
        \frac{1}{b}(x + n) \geq 1
    \]
    and choosing the least integer satisfying each. Furthermore, for all \(x \in \left[0,1\right)\), we have \(m = 0\). With this final substitution, we obtain the stated expression for \(Z_{1/b}Q_2\).
\end{proof}

\begin{figure}

\begin{tikzpicture}[scale=.8,every node/.style={font=\small}]

\begin{scope}[shift={(0,4)}]
\node[anchor=west] at (-6,-1.5) {\normalsize Case I: $\epsilon=0$ };

\draw[->] (0,-1.5) -- (12.6,-1.5) node[right] {$x$};
\draw[very thick, teal] (0,-1.5) -- (12,-1.5);
\node[below] at (0,-1.5) {$0$};
\node[below] at (12,-1.5) {$1$};
\fill[teal] (0,-1.5) circle (2pt);
\draw[teal] (12,-1.5) circle (2pt);

\draw[thick, teal, decorate,decoration={brace,amplitude=6pt}] (0,-1.25) -- (12,-1.25) node[midway, above=5pt] {$I_\epsilon$};
\end{scope}

\begin{scope}[shift={(0,2)}]
\node[anchor=west] at (-6,-1.5) {\normalsize Case II: $0<|\epsilon|<\tfrac12$};
\draw[->] (0,-1.5) -- (12.6,-1.5) node[right] {$x$};
\draw[very thick, orange!80!black] (0,-1.5) -- (3,-1.5);
\draw[very thick, teal] (3,-1.5) -- (9,-1.5);
\draw[very thick, orange!80!black] (9,-1.5) -- (12,-1.5);

\node[below] at (0,-1.5) {$0$};
\node[below] at (3,-1.5) {$|\epsilon|$};
\node[below] at (9,-1.5) {$1-|\epsilon|$};
\node[below] at (12,-1.5) {$1$};
\draw[thick, orange!80!black, decorate,decoration={brace,amplitude=6pt}] (0,-1.25) -- (3,-1.25) node[midway, above=5pt] {$I_\epsilon^{-}$};
\draw[thick, teal, decorate,decoration={brace,amplitude=6pt}] (3,-1.25) -- (9,-1.25) node[midway, above=5pt] {$I_\epsilon$};
\draw[thick, orange!80!black, decorate,decoration={brace,amplitude=6pt}] (9,-1.25) -- (12,-1.25) node[midway, above=5pt] {$I_\epsilon^+$};
\fill[orange!80!black] (0,-1.5) circle (2pt);
\fill[teal] (3,-1.5) circle (2pt);
\fill[orange!80!black] (9,-1.5) circle (2pt);
\draw[orange!80!black] (12,-1.5) circle (2pt);
\end{scope}

\begin{scope}[shift={(0,0)}]
\node[anchor=west] at (-6,-1.5) {\normalsize Case III: $|\epsilon|=\tfrac12$ };
\draw[->] (0,-1.5) -- (12.6,-1.5) node[right] {$x$};
\node[below] at (0,-1.5) {$0$};
\node[below] at (6,-1.5)  {$\tfrac12$};
\node[below] at(12,-1.5) {$1$};
 \draw[very thick, orange!80!black] (0,-1.5) -- (6,-1.5);
\draw[very thick, orange!80!black] (6,-1.5) -- (12,-1.5);
\draw[thick,orange!80!black, decorate,decoration={brace,amplitude=6pt}] (0,-1.25) -- (6,-1.25) node[midway, above=5pt] { $I_\epsilon^-$};
\draw[thick, orange!80!black, decorate,decoration={brace,amplitude=6pt}] (6,-1.25) -- (12,-1.25) node[midway, above=5pt] {$I_\epsilon^+$};
\fill[orange!80!black] (0,-1.5) circle (2pt);
\fill[orange!80!black] (6,-1.5) circle (2pt);
\draw[orange!80!black] (12,-1.5) circle (2pt);
\end{scope}
\end{tikzpicture}
\caption{The three cases of $\epsilon=F(b)$ in the region decomposition of $[0,1)$ underlying Theorem~\ref{thm:zerosZak}. Case II is the general case; Case I and Case III are its two limits, as $|\epsilon|\to0$ and $|\epsilon|\to\tfrac12$ respectively.
On $I_\epsilon: \ell=-N, n=N$, where $N=R(b)$.
on $I_\epsilon^+:\ell=-(\lfloor b\rfloor+1),  n=\lfloor b\rfloor$, and on  $I_\epsilon^-: \ell=-\lfloor b\rfloor,  n=\lfloor b\rfloor+1$.
}
\label{fig:eps-regions}
\end{figure}
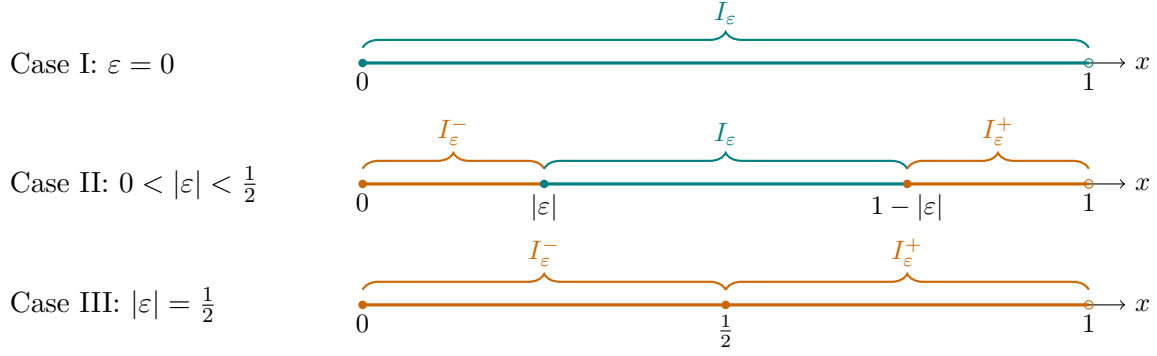

Figure \ref{fig:eps-regions} shows the partition of subintervals of $[0,1)$ in which $\ell$ and $n$ are constant. There are up to three subintervals depending on $|\epsilon|$:
\begin{enumerate}
    \item (Cases I and II): If $0\leq |\epsilon|<\tfrac{1}{2}$, define $I_\epsilon = [|\epsilon|, 1-|\epsilon|)$. Here $(\ell,  n) \equiv (-N,  N)$.
    \item (Cases II and III): If $0<|\epsilon|\leq \frac{1}{2}$, define $I_\epsilon^{-}:=[0, |\epsilon|)$, on which $(\ell,  n) \equiv (-\lfloor b\rfloor,  \lfloor b\rfloor+1)$ and $I_\epsilon^{+}:=[\,1-|\epsilon|,1)$, on which $(\ell,  n) \equiv (-(\lfloor b\rfloor+1),  \lfloor b\rfloor)$.
\end{enumerate}

For $n>0$, $\ell \neq 0$, $\omega=e^{-2\pi i \gamma}$, with $\gamma \neq 0$, we will use the following fact: by the triangle inequality, a nontrivial common solution to $\omega^\ell-2+\omega^n=0$ requires $\gcd(|\ell|,n)>1$. This only occurs for $x \in I_{\epsilon}$ (Cases I and II).

\subsection{Existence of roots of $Z_{1/b}Q_2$ for $\gcd(|\ell|, n)>1$, $\gamma\neq 0$}
The next lemma proves that  $(Z_{1/b}Q_2)(x,\gamma)=0$ for $\gcd(|\ell|,n)>1$ (equivalently $x\in I_\epsilon$) and $\gamma=k/N$ for $k\neq 0\mod N$.  

\begin{lemma}\label{lem:zerosZak1}
$(Z_{1/b}Q_2)(x,\frac kN)=0$ for all $x\in I_\epsilon$ and $k\not\equiv0\pmod N$.
\end{lemma}

\begin{proof}
For $x\in I_\epsilon$, $(\ell(x), n(x)) = (-N,  N)$. Since, $k\not \equiv 0\mod N$,  \eqref{b2zak} is well defined for $\gamma=k/N$, and becomes
\begin{align*}
    (Z_{1/b}Q_2)(x,k/N) &= b^{-3/2}\left[\frac{(b+x-N)\omega^{-N} - 2x - (b-x-N)\omega^{N}}{1-\omega} + \frac{\omega(\omega^{-N}-2+\omega^{N})}{(1-\omega)^2}\right]\\
    &= b^{-3/2}\left[\frac{(b+x-N) - 2x - (b-x-N)}{1-\omega} + \frac{\omega(1-2+1)}{(1-\omega)^2}\right]\\
    &=0,
\end{align*}
where $\omega=e^{-2\pi i k/N}$.

\end{proof}

\subsection{Explicits root of $Z_{1/b}Q_2$ for for $\gcd(|\ell|, n)=1$, $\gamma\neq 0$}

For $\gamma\neq 0$, it is apparent from formula \eqref{b2zak} that on any subset of $[0,1)$ with constant $(\ell,n)$, $(Z_{1/b}Q_2)(x,\gamma)$ is affine in $x$. If $x\in I_\epsilon^\pm$ (Cases II and III), then $(\ell,n)\in\{(-M,M{+}1),(-(M{+}1),M)\}$, so $\gcd(|\ell|,n)=\gcd(M,M{+}1)=1$; by the triangle inequality, this forces $\omega^\ell-2+\omega^n\ne0$ for every $\omega\ne1$. In this case, we can solve for the root using \eqref{b2zak}:
\begin{align}
    x' = -\frac{\omega}{1-\omega} - \frac{(b+\ell)\omega^{\ell}-(b-n)\omega^{n}}{\omega^{\ell}-2+\omega^{n}}, \qquad \omega=e^{-2\pi i \gamma}. \label{eq:xzero}
\end{align}
Notice that $x'$ is a genuine zero of $Z_{1/b}Q_2$ only if $x'=x$.

The goal of the next several results is to locate the root given by \eqref{eq:xzero} within the set $(x, \gamma)\in [0,1)\times(0,1)$.

\begin{definition}
Let $N$ be a nonnegative integer. The \emph{Dirichlet kernel} $D_N(\gamma)$ and \emph{Fej\'er kernel} $F_N(\gamma)$ are defined to be 
\[
    D_N(\gamma) = \sum_{k=-N}^N e^{2\pi ik\gamma}, \qquad F_N(\gamma) = \frac{1}{N}\sum_{k=0}^{N-1} D_k(\gamma).
\]

\end{definition}

\begin{lemma}[Properties of the Dirichlet and Fej\'er kernels]
For $N\in \mathbb{N}$, and $\gamma\in \mathbb{R}$, we have
\begin{align}
    G_N(\gamma) \defeq (N+1) F_{N+1}(\gamma) - D_N(\gamma) = \left(\frac{\sin(N\pi \gamma)}
    {\sin(\pi \gamma)}\right)^2 \geq 0, \label{eq:GNpos}
\end{align}
and
\begin{align}
    \sum_{k=1}^N k \cos(2\pi k \gamma) = \frac{N+1}{2}\left(D_N(\gamma) -
    F_{N+1}(\gamma)\right).\label{eq:Fn1cos}
\end{align}
\end{lemma}
\begin{proof}
To prove \eqref{eq:GNpos}, see that
\[
    \begin{aligned}
    (N+1) F_{N+1}(\gamma) - D_N(\gamma) &= \left(N+1 + 2\sum_{k=1}^N (N+1-k)
    \cos(2\pi k\gamma)\right) - \left(1 + 2\sum_{k=1}^N \cos(2\pi k\gamma)\right)\\
    &= N + 2\sum_{k=1}^N (N-k) \cos(2\pi k\gamma)\\
    &= \operatorname{Re}\left(N + 2\sum_{k=1}^N (N-k) e^{2\pi ik\gamma}\right)\\
    &= \operatorname{Re}\left(\frac{N(1-e^{4\pi i\gamma}) - 2e^{2\pi i\gamma}(1 - e^{2\pi i N\gamma})}
    {(1-e^{2\pi i\gamma})^2}\right)\\
    &= \frac{2 - 2\cos(2\pi \gamma) + \cos(2\pi (N-1)\gamma) - 2\cos(2\pi N\gamma) +
    \cos(2\pi (N+1)\gamma)}{3 - 4\cos(2\pi \gamma) + \cos(4\pi \gamma)}\\
    &= \frac{8\sin^2(N\pi \gamma) \sin^2(\pi \gamma)}{8\sin^4(\pi \gamma)}\\
    &= \left(\frac{\sin(N\pi \gamma)}{\sin(\pi \gamma)}\right)^2.
    \end{aligned}
\]
For \eqref{eq:Fn1cos}, see that
\[
    \begin{aligned}
    \frac{N+1}{2}(D_N(\gamma) - F_{N+1}(\gamma)) &= \frac{N+1}{2} \sum_{k=-N}^N
    e^{2\pi ik\gamma} - \frac{1}{2} \sum_{k=0}^N D_k(\gamma)\\
    &= \frac{N+1}{2}\left(1 + 2\sum_{k=1}^N \cos(2\pi k\gamma)\right) -
    \frac{1}{2}\left(N+1 + 2\sum_{k=1}^N (N+1-k) \cos(2\pi k\gamma)\right)\\
    &= \sum_{k=1}^N k\cos(2\pi k\gamma).
    \end{aligned}
\]
\end{proof}

\begin{lemma}\label{lem:noTwoZero}
    For $N>1$, no two of \(D_N\), \(F_{N+1}\), and \(G_N\) are simultaneously zero.
\end{lemma}
\begin{proof}
    The zero sets of \(D_N\), \(F_{N+1}\), and \(G_N\) are
    \[
        \begin{aligned}
       \mathcal{N}(D_N) &= \left\lbrace \frac{k}{2N+1} : k \not\equiv 0 \pmod{2N+1}\right\rbrace,\\
        \mathcal{N}(F_{N+1}) &= \left\lbrace \frac{k}{N+1} : k \not\equiv 0 \pmod{N+1} \right\rbrace,\\
        \mathcal{N}(G_N) &= \left\lbrace \frac{k}{N} : k \not\equiv 0 \pmod{N} \right\rbrace.
        \end{aligned}
    \]
    As \(\gcd(N, N+1) = \gcd(N, 2N+1) = \gcd(N+1, 2N+1) = 1\), these sets are disjoint.
\end{proof}

The following propositions characterizes all roots of $Z_{1/b}Q_2(\cdot, \gamma)$ for $\gamma\neq 0$ contained in intervals $ I_\epsilon^{\pm}$.

\begin{prop}\label{prop:boundary}
Let $b>1$ and $M=\lfloor b\rfloor$. Let $(x,\gamma)\in[0,1)\times(0,1)$ and $\omega=e^{-2\pi i\gamma}$.
\begin{enumerate}[(a)]
\item \textbf{$(\ell(x),n(x))=(-M,M{+}1)$.} Equation \eqref{eq:xzero} has a real solution $x'$ if and only if $\omega$ is an $M$th or $(M{+}1)$th root of unity; in that case,
\[
    x' = \begin{cases} b-M, & \text{if } \omega^M=1, \\ 1-(b-M), & \text{if } \omega^{M+1}=1, \end{cases}
\]
and $x'$ is a genuine zero of $Z_{1/b}Q_2$ only if $x'=x$.

\item \textbf{$(\ell(x),n(x))=(-(M{+}1),M)$.} Equation \eqref{eq:xzero} has a real solution $x'$ if and only if $\omega$ is an $M$th or $(M{+}1)$th root of unity; in that case,
\[
    x' = \begin{cases} 1-(b-M), & \text{if } \omega^M=1, \\ b-M, & \text{if } \omega^{M+1}=1, \end{cases}
\]
and $x'$ is a genuine zero of $Z_{1/b}Q_2$ only if $x'=x$.
\end{enumerate}
\end{prop}

\begin{proof}
Since $M,M{+}1$ are consecutive integers, $\gcd(M,M{+}1)=1$ and therefore since $\omega\neq 1$ it follows that $\omega^\ell-2+\omega^n\neq 0$. So \eqref{eq:xzero} applies directly for all $\gamma\in(0,1)$.

Taking $(\ell,n)=(-M,M{+}1)$ and writing \eqref{eq:xzero} as $A(\omega)/B(\omega)$, $x$ is real iff $A(\omega)\overline{B(\omega)}\in\mathbb R$, which reduces to

    \begin{align}
    \begin{split}
   \operatorname{Im}(A(\omega)\overline{B(\omega)}) &=6\sin(2\pi\gamma) + 2(M{+}1{-}b)\sin(2\pi(M{-}1)\gamma) + 2(b{-}M{+}2)\sin(2\pi M\gamma) \\&+ 2(b{-}M{-}3)\sin(2\pi(M{+}1)\gamma)
    \quad+ 2(M{-}b)\sin(2\pi(M{+}2)\gamma)\\ &+ 2(b{-}M{-}1)\sin(4\pi M\gamma) + 2(2M{+}1{-}2b)\sin(2\pi(2M{+}1)\gamma)\\
    &+ 2(b{-}M)\sin(4\pi(M{+}1)\gamma) = 0.\end{split}\label{eq:boundarytrigpoly}
    \end{align}
Since $Q_2$ is even, $Z_{1/b}Q_2(-x,\gamma)=Z_{1/b}Q_2(x,-\gamma)$, so \eqref{eq:boundarytrigpoly} governs $(\ell,n)=(-(M{+}1),M)$ identically.

 Under the substitution \(\sin(2\pi f\gamma) = \frac{\omega^f - \omega^{-f}}{2i}\), clearing denominators and the constant \(2i\), we obtain a polynomial in \(\omega\) whose roots on the unit circle correspond to \(\gamma\) admitting a real solution for \(x\)
    \[
        P(\omega)=2(\omega-1)(\omega^M-1)(\omega^{M+1}-1)Q(\omega),
    \]
    where
    \[
        Q(\omega)=(b{-}M)(\omega^{2M+2}{+}1)+(M{+}1{-}b)\omega(\omega^{2M}{+}1)-2\omega^{M+1}.
    \]
    
The polynomial $Q$ has a double root at $\omega=1$; factoring it out results in \[Q(\omega)/(\omega-1)^2=(b-M)D_M(\gamma)+G_M(\gamma).\]

Since $M=\lfloor b\rfloor$, $b-M\in(0,1)$: If \(D_M(\gamma) \geq 0\), then \((b - M) D_M(\gamma) + G_M(\gamma) > 0\), since \(G_M(\gamma) \geq 0\) and \(D_M\), \(G_M\) share no roots by Lemma~\ref{lem:noTwoZero}; If instead we assume \(D_M(\gamma) < 0\), then, as \(D_N(\gamma) + G_N(\gamma) = (N+1) F_{N+1}(\gamma) \geq 0\), it follows that \(G_N(\gamma) \geq - D_N(\gamma)\), and we have
        \[
            (b - M) D_M(\gamma) + G_M(\gamma) \geq (b - M - 1) D_M(\gamma) > 0.
        \] Therefore the only roots of $P(\omega)$ lying on the unit-circle are $\omega=1$ and the $M$th/$(M{+}1)$th roots of unity.

At $\omega^M=1$ with $(\ell,n)=(-(M{+}1),M)$:
\[
    (b+\ell)\omega^\ell-(b-n)\omega^n \Big|_{\omega^M=1} = (b-M-1)\omega^{-1}-(b-M),
\]
and $\omega^\ell-2+\omega^n=\omega^{-(M+1)}-1=\omega^{-1}-1$ using $\omega^M=1$. Substituting into \eqref{eq:xzero}:
\begin{align*}
        x & = -\frac{\omega}{1-\omega} - \frac{(b-M-1)\omega^{-1}-(b-M)}{\omega^{-1}-1}\\ &= -\frac{\omega}{1-\omega} - \frac{(b-M-1)-(b-M)\omega}{1-\omega}\\ &= -\frac{\omega+(b-M-1)-(b-M)\omega}{1-\omega} \\&= 1-(b-M).
\end{align*}
By reflection, the $M$th-root case for $(\ell,n)=(-M,M{+}1)$ gives $x=b-M$. The $(M{+}1)$th-root cases follow symmetrically.
\end{proof}

\begin{corollary}\label{cor:zerosZak12}
Suppose $\epsilon=-\tfrac12$. For $\gamma\in(0,1)$, $(Z_{1/b}Q_2)(\tfrac12,\gamma)=0$ if and only if
\[
    \gamma\in\{k/(N-1) : k\not\equiv0\pmod{N-1}\}\cup\{k/N : k\not\equiv0\pmod{N}\}.
\]
\end{corollary}

\begin{proof}
Since $|\epsilon|=\tfrac12$, $b-M=\tfrac12$, so $\{b-M,1-(b-M)\}=\{\tfrac12\}$. The claim follows directly from Proposition~\ref{prop:boundary}.
\end{proof}

\subsection{Characterization of all roots of $Z_{1/b}Q_2$ in $[0,1)^2$} \label{sec:charall}

\begin{proof}[Proof of Theorem \ref{thm:zerosZak}]
\begin{enumerate}
\item First, we claim that there are no real roots of $Z_{1/b}Q_2(\cdot, \gamma)$ for $\gamma\in\{0,1\}$. By Definition 1, $(Z_{1/b}Q_2)(x,0)=\sqrt{1/b}\sum_k Q_2(\tfrac1b(x-k))$, a sum of nonnegative terms. Since $Q_2$ is supported in $[-1,1]$ and $b>\tfrac12$, consecutive shifted windows overlap and collectively cover $\mathbb R$, so at least one term is strictly positive for every $x$; hence $(Z_{1/b}Q_2)(x,0)>0$ for all $x$. By periodicity, the statement holds for $\gamma=1$.


\item If $|\epsilon|=\tfrac12$, then $[|\epsilon|,1-|\epsilon|]=\{\tfrac12\}$. 
Combining this with Corollary~\ref{cor:zerosZak12}, we have that $(Z_{1/b}Q_2)(x, \gamma)=0$ if and only if $(x, \gamma)$ belongs to the two families of roots $\{x=1/2, \gamma=k/N, k\neq 0 \mod N\}$ and $\{ x=1/2, \gamma=k/(N-1), k\neq 0 \mod N -1\}$. 

The first family of roots satisfies \eqref{eq:zakrootsieps}, and we show that \eqref{eq:coteqn} holds in this case as well, i.e., \eqref{eq:coteqn} algebraically reduces to exactly this family, since $\cot(\pi\gamma)-\cot(N\pi\gamma)=0\iff\gamma=k/(N{-}1)$ for $k\neq 0 \mod N -1$.
To prove this, we have
\begin{align*}
\cot(\pi\gamma) - \cot(N\pi\gamma) 
&= \frac{\cos(\pi\gamma)}{\sin(\pi\gamma)} - \frac{\cos(N\pi\gamma)}{\sin(N\pi\gamma)} \\
&= \frac{\cos(\pi\gamma)\sin(N\pi\gamma) - \cos(N\pi\gamma)\sin(\pi\gamma)}{\sin(\pi\gamma)\sin(N\pi\gamma)} \\
&= \frac{\sin(N\pi\gamma - \pi\gamma)}{\sin(\pi\gamma)\sin(N\pi\gamma)} = \frac{\sin((N-1)\pi\gamma)}{\sin(\pi\gamma)\sin(N\pi\gamma)},
\end{align*}
using $\sin(A-B)=\sin A\cos B-\cos A\sin B$ with $A=N\pi\gamma$, $B=\pi\gamma$. For $\gamma\in(0,1)\setminus\{k/N:k\in\mathbb Z\}$ (so that $\sin(\pi\gamma),\sin(N\pi\gamma)\ne0$ and the expression is defined), this vanishes if and only if
\[
\sin((N-1)\pi\gamma)=0 \iff (N-1)\pi\gamma = m\pi \text{ for some } m\in\mathbb Z \iff \gamma=\frac{m}{N-1}.
\]
It remains to check $\gamma=m/(N-1)$ avoids the excluded poles. Since $\gcd(N-1,N)=1$, if $m/(N-1)=j/N$ for integers $m,j$, then $mN=j(N-1)$, so $(N-1)\mid mN$, and since $\gcd(N-1,N)=1$, $(N-1)\mid m$; combined with $m\not\equiv0\pmod{N-1}$ (needed for $\gamma\ne0,1$), this is impossible. So every $\gamma=k/(N-1)$ with $k\not\equiv0\pmod{N-1}$ avoids both poles automatically, and
\[
\cot(\pi\gamma)-\cot(N\pi\gamma)=0 \iff \gamma=\frac{k}{N-1},\ k\not\equiv0\pmod{N-1}.
\]

This proves Theorem \ref{thm:zerosZak} for the case $\epsilon=-1/2$.

\item Next, for $0\leq|\epsilon|< \tfrac12$, we eliminate roots outside of $x\in[|\epsilon|,1-|\epsilon|]$.

Suppose $x \not\in I_{\epsilon}$. Then, $\gcd(|\ell|,n)=1$, so there are no degenerate roots since $\omega^\ell-2+\omega^n\ne0$, and therefore \eqref{eq:xzero} applies. 

On $I_\epsilon^{+}$ $(\ell, n)=(-(M+1), M)$,
\begin{itemize}
\item $\epsilon>0\implies M=N$
\begin{align*}
    x &= -\frac{\omega}{1-\omega} - \frac{(b-(N+1))\omega^{-(N+1)}-(b-(N))\omega^{N}}{\omega^{-(N+1)}-2+\omega^{N}}\\
&= -\frac{\omega}{1-\omega} - \frac{(b-N-1)\omega^{-1}-(b-N)}{\omega^{-1}-1}\\
&= -\frac{\omega}{1-\omega} - \frac{(b-N-1)(1-\omega) - \omega}{1-\omega}\\
&= 1-\epsilon.
\end{align*}

\item $\epsilon<0\implies M=N-1$
\begin{align*}
    x &= -\frac{\omega}{1-\omega} - \frac{(b-(N))\omega^{-(N)}-(b-(N-1))\omega^{M}}{\omega^{-(N)}-2+\omega^{N-1}}\\
     &= -\frac{\omega}{1-\omega} + \frac{(b-N+1)(\omega^{-1}-1) +1}{-1+\omega^{-1}}\\
     &= 1+\epsilon.
\end{align*}
\end{itemize}
On $I_\epsilon^{-}$, $(\ell, n) = (-(M), M+1)$ and we see that in both cases the root is $x=\epsilon \not\in I_\epsilon^{-}$: 
\begin{itemize}
    \item$\epsilon>0\implies M=N$:
\begin{align*}
    x &= -\frac{\omega}{1-\omega} - \frac{(b-N)-(b-(N+1))\omega}{-1+\omega}\\
    &= -\frac{\omega}{1-\omega} - \frac{(b-(N+1))-(b-(N+1))\omega +1}{-1+\omega}\\
    &= -\frac{\omega}{1-\omega} - \frac{(b-(N+1))(1-\omega) +1}{-1+\omega}\\
     &=\epsilon.
\end{align*}
\item $\epsilon<0\implies M=N-1$:
\begin{align*}
    x &= -\frac{\omega}{1-\omega} - \frac{(b-(N-1))\omega^{-(N-1)}-(b-(N))\omega^{N}}{\omega^{-(N-1)}-2+\omega^{N}}\\
    &= -\frac{\omega}{1-\omega} - \frac{(b-N+1)\omega-(b-N+1) + 1}{\omega-1}\\
     &= -\frac{\omega}{1-\omega} - \frac{(b-N+1)(\omega-1) + 1}{\omega-1}\\
     &= -\epsilon.
\end{align*}
\end{itemize}

This proves that the only roots that lie outside of $I_\epsilon\times(0,1)$ have the form $x=1-|\epsilon|$, $\gamma=k/N$ for $k\not\equiv0\pmod N$: evaluating at $\gamma=k/N$ gives a genuine root only on $I_\epsilon^+$ (where it lands at $x=1-|\epsilon|$), while the same $\gamma$ evaluated on $I_\epsilon^-$ gives $x=\epsilon\notin I_\epsilon^-$, so no root occurs there.

\item Now we characterize roots $I_\epsilon\times (0,1)$ for $0\leq |\epsilon|<\tfrac12$ and show that there are exactly two families of roots: degenerate roots satisfying \eqref{eq:zakrootsieps} and nondegenerate roots satisfying \eqref{eq:coteqn}.

By the triangle inequality, the only degenerate roots can occur for $\gamma=k/N, k\not\equiv0\pmod N$. These roots, given by Lemma~\ref{lem:zerosZak1} are exactly the set $\{x\in I_\epsilon, \gamma=\frac{k}{N}, k\neq 0 \mod N\}$.

For nondegenerate roots $(x,\gamma)$, with $\omega^{-N}-2+\omega^N\ne0$, \eqref{eq:xzero} applies on $I_\epsilon$ with $(\ell,n)=(-N,N)$:
\[
    x = -\frac{\omega}{1-\omega} - \frac{(b-N)(\omega^{-N}-\omega^N)}{\omega^{-N}-2+\omega^N}.
\]
Write $\omega=e^{-i\theta}$, $\theta=2\pi\gamma$. Using $\dfrac{\omega}{1-\omega}=\dfrac{e^{-i\theta/2}}{2i\sin(\theta/2)}=-\dfrac12-\dfrac i2\cot(\theta/2)$,
\[
    -\frac{\omega}{1-\omega} = \frac12 + \frac i2\cot(\pi\gamma).
\]
For the second term, with $\phi=N\theta$: $\omega^{-N}-\omega^N=2i\sin\phi$ and $\omega^{-N}-2+\omega^N=-4\sin^2(\phi/2)$, so
\[
    -\frac{(b-N)(\omega^{-N}-\omega^N)}{\omega^{-N}-2+\omega^N} = (b-N)\,\frac{i\sin\phi}{2\sin^2(\phi/2)} = i(b-N)\cot(N\pi\gamma).
\]
Combining,
\[
    x = \frac12 + i\left[\frac12\cot(\pi\gamma)+(b-N)\cot(N\pi\gamma)\right].
\]

So $x$ is real if and only if the bracket vanishes, i.e.\ exactly \eqref{eq:coteqn}, in which case $x=\tfrac12$, which always lies in $[|\epsilon|,1-|\epsilon|]$.
\end{enumerate}

Therefore, we have shown that all roots $(x,\gamma)$ of $Z_{1/b}Q_2$ satisfy $x\in [|\epsilon|, 1-|\epsilon|]$ and satisfy either \eqref{eq:zakrootsieps} or \eqref{eq:coteqn}.

\end{proof}



    

\subsection{Alternate proof of Theorem \ref{thm:lemneil} for $n=2$}\label{sec:alternateproof}
\begin{proof}
       Fixing \(r = 0\), we obtain the row vector $\phi = (\phi_0, \hdots, \phi_{q-1})$, $\phi_s = p^{-1/2}Z_{1/b} Q_2(x - s\tfrac pq, \gamma)$, of the Zibulski-Zeevi matrix. Then, for $(x,\gamma)\in [0,1)^2$,
   \begin{align}
        \norm{\phi(x,\gamma)}^2 = \frac1p\sum_{s=0}^{q-1} \abs{Z_{\frac{1}{b}} Q_2\!\left(x - s\tfrac pq,\ \gamma\right)}^2.  \label{eq:phinorm}
   \end{align}
    If \eqref{eq:phinorm} vanishes at some $(x_0,\gamma_0)$, then $\Phi_{0,s}=0$ for every $s$ there, so $\Phi(x_0,\gamma_0)$ has an identically zero row and $\{\phi_s^{Q_2}(x_0,\gamma_0)\}$ cannot span $\mathbb C^p$. Since $\Phi$ is continuous, this forces $\lambda_{\min}(\Phi\Phi^*)$ to be arbitrarily small on every neighborhood of $(x_0,\gamma_0)$, hence its essential infimum is zero, so by Theorem~\ref{thm:ZZchar}, $\mathcal G(Q_2,a,b)$ is not a frame.

    Since $|Z_{1/b}Q_2|^2$ is $1$-periodic in $x$, each term of \eqref{eq:phinorm} can be evaluated at $x-s\tfrac pq \bmod 1 \in[0,1)$. 
    Since $\gcd(p,q)=1$, the map $s\mapsto sp\bmod q$ is a bijection on $\{0,\ldots,q-1\}$, so $\{s\tfrac pq\bmod1 : s=0,\ldots,q-1\}=\{0,\tfrac1q,\ldots,\tfrac{q-1}q\}$ as sets, independent of $p$. For $b=N+\epsilon$ with $|\epsilon|\leq \tfrac12$, suppose $|\epsilon|\le\tfrac{1}{2q}$; we show every term of \eqref{eq:phinorm} vanishes at $\gamma_0=k/N$ ($k\not\equiv0\pmod N$) and $x_0=1-|\epsilon|$. For each $s=0,\ldots,q-1$,
\[
    x_0 - \tfrac sq = 1-|\epsilon|-\tfrac sq \;\ge\; 1-\tfrac1{2q}-\tfrac{q-1}q \;=\;\tfrac{1}{2q}\;>\;0.
\]
This value lies in $[|\epsilon|,1-|\epsilon|]$ iff $|\epsilon|\le\tfrac{q-s}{2q}$; since $|\epsilon|\le\tfrac1{2q}\le\tfrac{q-s}{2q}$ for every $s\le q-1$, this holds for every $s$. So every term of \eqref{eq:phinorm} vanishes at $(x_0,\gamma_0)$, by  Theorem~\ref{thm:zerosZak}.

\end{proof}

\section{Sharpness of Lemvig--Nielsen obstructions}\label{sec:main}

\begin{proof}[Proof of Theorem \ref{thm:main}]

    Suppose $|\epsilon|=|F(b)|>\frac{1}{2q}$. Fix $(x, \gamma)\in [0, 1)\times (0,1)$.

    There are $q$ terms in the sum \eqref{eq:phinorm} that evaluate the same function $Z_{\frac{1}{b}}Q_2(\cdot, \gamma)$ at the $q$ distinct, equally spaced points
    \[x- \frac{s}{q}\mod 1, \qquad s=0, \hdots, q-1.\]
These points correspond to points on the circle of circumference $1$.

Fix $\gamma\in(0,1)$. Theorem \ref{thm:zerosZak} implies that $Z_{\frac{1}{b}}Q_2(\cdot, \gamma)$ has zeros in $[0,1)$ if either 
\begin{enumerate}[(a)]
    \item $\gamma=k/N$ for $k\neq 0\mod N$, or
    \item $\gamma$ satisfies \eqref{eq:coteqn}.
\end{enumerate}
Suppose (a) is true. Then, since $Z_{\frac{1}{b}}Q_2(\cdot, k/N)$ vanishes on $\overline{I_\epsilon}$, it follows that the sum \eqref{eq:phinorm} is zero if and only if $x-s\tfrac{1}{q} \in \overline{I_\epsilon}$ for all $s=0,\hdots q-1$. This means that $x\in \cap_{s=0}^{q-1} (\overline{I_\epsilon}+ s\tfrac{1}{q})$. Since $\overline{I_\epsilon}$ is an interval of length $(1-2|\epsilon| )$ it follows that $[0,1)\backslash I_\epsilon$ has length $2|\epsilon| $. 

The translates of the complement, $(I_\epsilon^c+s\tfrac1q)$ where $I_\epsilon^c=[0,1)\setminus\overline{I_\epsilon}$, have total length $q\cdot2|\epsilon|$ and cover the unit circle when $2|\epsilon|\ge\tfrac1q$, i.e.\ $|\epsilon|\ge\tfrac1{2q}$. Since $|\epsilon|>\tfrac1{2q}$ by hypothesis, the complements' translates cover the circle, so by De Morgan $\bigcap_{s=0}^{q-1}(\overline{I_\epsilon}+s\tfrac1q)=\emptyset$. Therefore there is no $x$ that makes all of the terms in \eqref{eq:phinorm} zero.

In case $(b)$, suppose that $\gamma=\gamma_0$ is a solution of \eqref{eq:coteqn}. Then, by Theorem \ref{thm:zerosZak}, it follows that $x=\frac{1}{2}$. In order for all of the terms in \eqref{eq:phinorm} to vanish, we need $x-\tfrac{s}{q}\equiv \frac{1}{2}\mod 1$ for all $s=0,\hdots, q-1.$ Since $x-\tfrac{s}{q}$ takes $q$ distinct values, with $q\geq 2$ this means there is a nonzero term in \eqref{eq:phinorm}. 

Therefore, \eqref{eq:phinorm} is positive for all $(x,\gamma)\in \mathbb{R}/\mathbb{Z}\times (0,1)$. Since $Q_2$ is continuous and compactly supported, $Z_{\frac{1}{b}}Q_2$ is continuous in $(x,\gamma)$, and $\|\phi(x,\gamma)\|^2$ is continuous on the compact set $[0,1]\times[0,1]$. Since a continuous, everywhere positive function on a compact set attains a positive minimum, it follows that $\inf_{(x,\gamma)}\|\phi(x,\gamma)\|^2>0$. Therefore, by Theorem \ref{thm:ZZchar}, $\mathcal{G}(Q_2,a,b)$ is a frame.

Therefore, all other points along the curves $ab=\frac{1}{q}$ which do not belong to the
    Lemvig--Nielsen obstructions are in fact in $\mathcal{F}(Q_2, a,b)$.
\end{proof}

\section{Acknowledgements} This work was sponsored by NSF DMS-2309651. {\it Declaration of Generative AI and AI-assisted technologies in the writing process:}
During the preparation of this work, the authors used AI (including Claude and DeepSeek) to assist with numerical and symbolic verification of mathematical claims, and to assist with exposition, LaTeX formatting, and figure preparation. All mathematical content, including every claim identified or verified with the tool's assistance, was independently reviewed and confirmed by the authors, who take full responsibility for the correctness and content of the published work.

\nocite{*}

\printbibliography

\end{document}